\documentclass[12pt,fleqn]{amsart}

\usepackage[paperwidth=210mm,paperheight=297mm,inner=3.5cm,outer=3.5cm,top=2.5cm,bottom=2.6cm]{geometry}

\usepackage{amsmath}
\usepackage{amsthm}
\usepackage{amssymb}
\usepackage{tikz}
\usetikzlibrary{patterns}
\usepackage{pgfplots}
\usepackage{mathrsfs}
\usepackage[colorlinks=false]{hyperref}
\usepackage{paralist}

\theoremstyle{plain}
\newtheorem{Theorem}{Thm}[section]
\newtheorem{Thm}[Theorem]{Theorem}
\newtheorem{Lem}[Theorem]{Lemma}
\newtheorem{Cor}[Theorem]{Corollary}
\newtheorem{Prop}[Theorem]{Proposition}
\newtheorem*{Thm*}{Theorem}
\theoremstyle{definition}
\newtheorem{Def}[Theorem]{Definition}
\newtheorem{Rem}[Theorem]{Remark}
\newtheorem{Exm}[Theorem]{Example}

\makeatletter
\renewenvironment{proof}[1][\proofname]{\par
\pushQED{\qed}%
\normalfont \topsep6\p@\@plus6\p@\relax
\trivlist
\item[\hskip\labelsep
\scshape
#1\@addpunct{.}]\ignorespaces
}{%
\popQED\endtrivlist\@endpefalse
}
\makeatother

\newcommand\mbb{\mathbb}

\newcommand\mcal{\mathcal}

\newcommand\A{\mbb{A}}
\newcommand\C{\mbb{C}}
\newcommand\N{\mbb{N}}
\renewcommand\P{\mbb{P}}
\newcommand\Q{\mbb{Q}}
\newcommand\R{\mbb{R}}

\newcommand\V{\mcal{V}}

\newcommand\ol{\overline}
\newcommand\wh{\widehat}

\newcommand\into{\mapsto}
\DeclareMathOperator\topint{int}
\DeclareMathOperator\ev{ev}
\DeclareMathOperator\T{T}

\DeclareMathOperator\ex{Ex}
\DeclareMathOperator\exr{Exr}
\DeclareMathOperator\codim{codim}
\DeclareMathOperator\co{co}

\newcommand\DX[1][X]{#1^\ast}
\newcommand\DP[1][n]{(\P^{#1})^\ast}
\newcommand\DV[1][V]{#1^\ast}
\newcommand\CN[1][X]{{\rm{CN}(#1)}}

\newcommand\Xreg[1][X]{#1_{{\rm{reg}}}}
\newcommand\Xsing[1][X]{#1_{{\rm{sing}}}}

\begin{document}
\title{Algebraic Boundaries of Convex Semi-algebraic Sets}
\author{Rainer Sinn}
\address{Georgia Institute of Technology, Atlanta, USA}
\email{sinn@math.gatech.edu}
\date{}

\begin{abstract}
We study the algebraic boundary of a convex semi-algebraic set via duality in convex and algebraic geometry. We generalize the correspondence of facets of a polytope with the vertices of the dual polytope to general semi-algebraic convex sets. In this case, exceptional families of extreme points might exist and we characterize them semi-algebraically. We also give an algorithm to compute a complete list of exceptional families, given the algebraic boundary of the dual convex set.
\end{abstract}

\maketitle

\section{Introduction}
The algebraic boundary of a semi-algebraic set is the smallest algebraic variety containing its boundary in the euclidean topology. For a full-dimensional polytope $\R^n$, it is the hyperplane arrangement associated to its facets which has been studied extensively in discrete geometry and complexity theory in linear programming \cite{LoeStVinMR2946462}.
The algebraic boundary of a convex set which is not a polytope has recently been considered in other special cases, most notably the convex hull of a variety by Ranestad and Sturmfels, cf.~\cite{RanSt} and \cite{RanStMR2911165}. This class includes prominent families such as the moment matrices of probability distributions and the highly symmetric orbitopes. It does not include examples such as hyperbolicity cones and spectrahedra, which have received attention from applications of semi-definite programming in polynomial optimisation, see \cite{SIAMbook} and \cite{VinnMR2962792}, and statistics of Gaussian graphical models, see \cite{StuUhlMR2652308}. 

First steps towards using the algebraic boundary of a spectrahedron for a complexity analysis of semi-definite programming have been taken by Nie, Ranestad, and Sturmfels \cite{NieRanStMR2546336}. For semi-definite liftings of convex semi-algebraic sets via Lasserre relaxations or theta body construction, the singularities of the algebraic boundary on the convex set give obstructions, cf.~\cite{NetPlauSchMR2600247}, \cite{GouNetMR2837559}.

So algebraic boundaries are central objects in applications of algebraic geometry to convex optimisation and statistics. In this paper, we want to consider the class of all sets for which the algebraic boundary is an algebraic hypersurface: convex semi-algebraic sets with non-empty interior.
Our goal in this paper is to extend the study of the algebraic boundary of the convex hull of a variety started by Ranestad and Sturmfels in \cite{RanSt} and \cite{RanStMR2911165} to general convex semi-algebraic sets. The most natural point of view in the general setting is via convex duality and its algebraic counterpart in projective algebraic geometry. The first main theorem generalizes and implies the correspondence between facets of a polytope with vertices of its dual polytope.

\begin{Thm*}[Corollary \ref{Thm:DualIrrCompAffine}]
Let $C\subset\R^n$ be a compact convex semi-algebraic set with $0\in\topint(C)$. Let $Z$ be an irreducible component of the Zariski closure of the set of extreme points of its dual convex body. Then the variety dual to $Z$ is an irreducible component of the algebraic boundary of $C$.
\end{Thm*}
For polytopes, this theorem is the whole story. In the general semi-algebraic case, not every irreducible component of the algebraic boundary of $C$ arises in this way, as we will see below. We study the exceptional cases and give a complete semi-algebraic description of the exceptional families of extreme points in terms of convex duality (normal cones) and a computational way of getting a list of potentially exceptional strata from the algebraic boundary of the dual. This proves an assertion made by Sturmfels and Uhler in \cite[Proposition 2.4]{StuUhlMR2652308}.

The main techniques come from the duality theories in convex and projective algebraic geometry.
For an introduction to convex duality, we refer to Barvinok's textbook \cite{BarvinokMR1940576}. The duality theory for projective algebraic varieties is developed in several places, e.g.~Harris \cite{HarrisiMR1182558}, Tevelev \cite{TevMR2113135}, or Gelfand-Kapranov-Zelevinsky \cite{GelMR2394437}.

This article is organized as follows: In Section \ref{sec:Prelim}, we introduce the algebraic boundary of a semi-algebraic set and discuss some special features of convex semi-algebraic sets coming from their algebraic boundary. The section sets the technical foundation for Section \ref{sec:AlgBound}, where we prove the main results of this work.

\section{The Algebraic Boundary and Convexity}\label{sec:Prelim}
This section is supposed to be introductory. We will fix notation and observe some basic features of convex semi-algebraic sets, their algebraic boundary, and some special features relying on this algebraic structure. The main results will be proven in the following section.
\begin{Def}
Let $S\subset \R^n$ be a semi-algebraic set.
The \emph{algebraic boundary} of $S$, denoted as $\partial_a S$, is the Zariski closure in $\A^n$ of the euclidean boundary of $S$.
\end{Def}

\begin{Rem}
In this paper, we fix a subfield $k$ of the complex numbers. The most important choices to have in mind are the reals, the complex numbers or the rationals. When we say Zariski closure, we mean with respect to the $k$-Zariski topology, i.e.~the topology on $\C^n$ (resp.~$\P(\C^{n+1})$) whose closed sets are the algebraic sets defined by polynomials (resp.~homogeneous polynomials) with coefficients in $k$. The set $\C^n$ (resp.~$\P(\C^{n+1})$) equipped with the $k$-Zariski topology is usually denoted $\A^n_k$ (resp.~$\P^n_k$). We drop the field $k$ in our notation. The statements in this paper are true over any subfield $k$ of the complex numbers given that the semi-algebraic set in consideration can be defined by polynomial inequalities with coefficients in $k\cap \R$.

If we are interested in symbolic computation, we tend to consider semi-algebraic sets defined by polynomial inequalities with coefficients in $\Q$ and take Zariski closures in the $\Q$-Zariski topology.
\end{Rem}

We first want to establish that the algebraic boundary of a convex body is a hypersurface.
\begin{Def}
A subset of $\R^n$ is called \emph{regular} if it is contained in the closure (in the euclidean topology) of its interior.
\end{Def}

\begin{Rem}
Every convex semi-algebraic set with non-empty interior is regular and the complement of a convex semi-algebraic set is also regular.
\end{Rem}

\begin{Lem}\label{Lem:AlgBoundHyp}
Let $\emptyset\neq S\subset\R^n$ be a regular semi-algebraic set and suppose that its complement $\R^n\setminus S$ is also regular and non-empty.
Each irreducible component of the algebraic boundary of $S$ has codimension $1$ in $\A^n$, i.e.~$\partial_a S$ is a hypersurface.
\end{Lem}

\begin{proof}
By Bochnak-Coste-Roy \cite[Proposition 2.8.13]{BochnakMR1659509}, $\dim(\partial S)\leq n-1$. Conversely, we prove that the boundary $\partial S$ of $S$ has local dimension $n-1$ at each point $x\in \partial S$: Let $x\in\partial S$ be a point and take $\epsilon >0$. Then $\topint(S)\cap {\rm B}(x,\epsilon)$ and $\topint(\R^n\setminus S)\cap {\rm B}(x,\epsilon)$ are non-empty, because both $S$ and $\R^n\setminus S$ are regular. Applying \cite[Lemma 4.5.2]{BochnakMR1659509}, yields that
\[
\dim(\partial S\cap {\rm B}(x,\epsilon)) = \dim({\rm B}(x,\epsilon)\setminus (\topint(S)\cup (\R^n\setminus \ol{S}))) \geq n-1
\]
Therefore, all irreducible components of $\partial_a S={\rm cl}_{Zar}(\partial S)$ have dimension $n-1$.
\end{proof}

\begin{Exm}
The assumption of $S$ being regular cannot be dropped in the above lemma. Write $h:=x^2+y^2+z^2-1\in\R[x,y,z]$. Let $S$ be the union of the unit ball with the first coordinate axis, i.e.~$S = \{(x,y,z)\in\R^3\colon y^2h(x,y,z)\leq 0,z^2h(x,y,z)\leq 0\}$. The algebraic boundary of $S$ is the union of the sphere $\V(h)$ and the line $\V(y,z)$, which is a variety of codimension $1$ with a lower dimensional irreducible component.
\end{Exm}

\begin{Rem}
In the above proof of Lemma \ref{Lem:AlgBoundHyp}, we argue over the field of real numbers. The algebraic boundary of $S$, where the Zariski closure is taken with respect to the $k$-Zariski topology for a different field $k$, is also a hypersurface. It is defined by the reduced product of the Galois conjugates of the polynomial defining $\partial_a S$ over $\R$, whose coefficients are algebraic numbers over $k$.
\end{Rem}

\begin{Cor}
Let $C\subset \R^n$ be a compact semi-algebraic convex set with non-empty interior. Its algebraic boundary is a hypersurface.\qed
\end{Cor}

This property characterises the semi-algebraic compact convex sets.
\begin{Prop}
A compact convex set with non-empty interior is semi-algebraic if and only if its algebraic boundary is a hypersurface.
\end{Prop}

\begin{proof}
The converse follows from results in semi-algebraic geometry. Namely if the algebraic boundary $\partial_a C$ is an algebraic hypersurface, its complement $\R^n\setminus (\partial_a C)(\R)$ is a semi-algebraic set and the closed convex set $C$ is the closure of the union of finitely many of its connected components. This is semi-algebraic by Bochnak-Coste-Roy \cite[Proposition 2.2.2 and Theorem 2.4.5]{BochnakMR1659509}.
\end{proof}

By the construction of homogenisation in convexity, the algebraic boundary of a pointed and closed convex cone relates to the algebraic boundary of a compact base via the notion of affine cones in algebraic geometry.
\begin{Rem}\label{Rem:ConeProj}
Let $C\subset \R^n$ be a compact semi-algebraic convex set and let $\co(C)\subset \R\times \R^n$ be the convex cone over $C$ embedded at height $1$, i.e.~$\co(C) = \{(\lambda,\lambda x)\colon \lambda \geq 0, x\in C\}$. Since a point $(1,x)$ lies in the boundary of $\co(C)$ if and only if $x$ is a boundary point of $C$, the affine cone $\{(\lambda,\lambda x)\colon \lambda\in\C, x\in \partial_a C\}$ over the algebraic boundary of $C$ is a constructible subset of the algebraic boundary of $\co(C)$. More precisely, we mean that $\partial_a \co(C)= \wh{X}$, where $X$ is the projective closure of $\partial_a C$ with respect to the embedding $\A^n\into \P^n$, $(x_1,\ldots,x_n)\mapsto (1:x_1:\ldots:x_n)$.
\end{Rem}

Recall that a closed convex cone $C\subset \R^n$ is called pointed if $C\cap (-C) = \{0\}$, i.e.~it does not contain a line.
\begin{Cor}
Let $C\subset \R^{n+1}$ be a pointed closed semi-algebraic convex cone. Its algebraic boundary is a hypersurface in $\A^{n+1}$ and an algebraic cone.
In particular, it is the affine cone over its projectivisation in $\P^{n}$, i.e.
\[
\pushQED{\qed}
\wh{\P\partial_a C} = \partial_a C. \qedhere
\popQED
\] 
\end{Cor}

We will now take a look at convex duality for semi-algebraic sets. Given a compact convex set $C\subset \R^n$, we write $C^o = \{\ell\in(\R^n)^\ast\colon \ell(x)\geq -1 \text{ for all } x\in C\}$ for the dual convex set. We use the notation $\Xreg$ for the set of all regular (or smooth) points of an algebraic variety $X$.
\begin{Prop}\label{Prop:ExpExtPoi}
Let $C\subset \R^n$ be a compact semi-algebraic convex set with $0\in\topint(C)$ and set $S:=\partial C^o \cap \Xreg[(\partial_a C^o)]$. For every $\ell\in S$, the face supported by $\ell$ is a point. The set $S$ is an open and dense (in the euclidean topology) semi-algebraic subset of the set $\partial C^o$ of all supporting hyperplanes to $C$.
\end{Prop}

\begin{proof}
If $\ev_x$ is a supporting hyperplane to $C^o$ at $\ell$, then $\ell(x)=-1$ and $C^o$ lies in one halfspace defined by $\ev_x$. Therefore, $(\partial_a C^o)(\R)$ lies locally around $\ell$ in one halfspace defined by $\ev_x$ and so $\ev_x$ defines the unique tangent hyperplane to $\partial_a C^o$ at $\ell$.
Now we show that $x$ is an extreme point of $C$, exposed by $\ell$. Suppose $x=\frac12 (y+z)$ with $y,z\in C$, then $\ell(y)= -1$ and $\ell(z)= -1$. Since $y$ and $z$ are, by the same argument as above, also normal vectors to the tangent hyperplane $\T_\ell \partial_a C^o$, we conclude $x=y=z$.
\end{proof}

The same statement is true for convex cones: We denote the dual convex cone to $C\subset\R^{n+1}$ as $C^\vee = \{\ell\in(\R^{n+1})^\ast\colon \ell(x)\geq 0 \text{ for all }x\in C\}$.
\begin{Cor}\label{Cor:ExpExtPoi}
Let $C\subset\R^{n+1}$ be a pointed closed semi-algebraic convex cone with non-empty interior and set $S:=\partial C^\vee\cap\Xreg[(\partial_a C^\vee)]$. For every $\ell\in S$, the face supported by $\ell$ is an extreme ray of $C$. The set $S$ is open and dense (in the euclidean topology) semi-algebraic subset of $\partial C^\vee$. \qed
\end{Cor}

\begin{Exm}\label{Exm:ExPIncMax}
(a) In the case that $C$ is a polytope, the set $S$ of regular points of the algebraic boundary is exactly the set of linear functionals exposing extreme points. Indeed, in this case the algebraic boundary of $C$ is a union of affine hyperplanes, namely the affine span of its facets. A point in $\partial C$ is a regular point of the algebraic boundary $\partial_a C$ if and only if it lies in the relative interior of a facet, cf.~Barvinok \cite[Theorem VI.1.3]{BarvinokMR1940576}. These points expose the vertices of $C^o$.\\
(b) In general, a linear functional $\ell\in\partial C^o$ exposing an extreme point of $C$ does not need to be a regular point of the algebraic boundary of $C^o$ as the following example shows:
Let $C$ be the convex set in the plane defined by the inequalities $y\geq (x+1)^2-3/2$, $y\geq (x-1)^2-3/2$ and $y\leq 1$. Consider the extreme point $x=(0,-1/2)$ of $C$. The dual face is the line segment between the vectors $(-2,1)$ and $(2,1)$, the normal vectors to the tangent lines to the curves defined by $y-(x+1)^2+3/2$ and $y-(x-1)^2+3/2$, which meet transversally in $x$.
Indeed, the linear functionals $(-2,1)$ and $(2,1)$ both expose extreme points; but they are each intersection points of a line and a quadric in the algebraic boundary of $C^o$ and so they are singular points of $\partial_a C^o$.
\end{Exm}

The extreme points (resp.~rays) of a convex set play an important role for duality. They will also be essential in a description of the algebraic boundary using the algebraic duality theory. So we fix the following notation:
\begin{Def}
(a) Let $C\subset\R^n$ be a convex semi-algebraic set. We denote by $\ex_a(C)$ the Zariski closure of the union of all extreme points of $C$ in $\A^n$.\\
(b) Let $C\subset\R^{n+1}$ be a semi-algebraic convex cone. We write $\exr_a(C)$ for the Zariski closure of the union of all extreme rays of $C$ in $\A^{n+1}$.
\end{Def}

\begin{Rem}\label{Rem:ExtPointsSemiAlg}
(a) Note that the union of all extreme points of a convex semi-algebraic set is a semi-algebraic set by quantifier elimination because the definition is expressible as a first order formula in the language of ordered rings, cf.~Bochnak-Coste-Roy \cite[Proposition 2.2.4]{BochnakMR1659509}. Therefore, its Zariski closure is an algebraic variety whose dimension is equal to the dimension of $\ex(C)$ as a semi-algebraic set, cf.~Bochnak-Coste-Roy \cite[Proposition 2.8.2]{BochnakMR1659509}. Of course, the same is true for convex cones and the Zariski closure of the union of all extreme rays.\\
(b) Note that $\exr_a(C)$ is an algebraic cone. In particular, we have
\[
\exr_a(C) = \wh{\P\exr_a(C)}.
\]
\end{Rem}

\begin{Lem}
Let $C\subset \R^n$ be a compact semi-algebraic convex set with $0\in\topint(C)$. For a general extreme point $x\in \ex_a(C)$ there is a supporting hyperplane $\ell_0\in\partial C^o$ exposing the face $x$ and a semi-algebraic neighbourhood $U$ of $\ell_0$ in $\partial C^o$ such that every $\ell\in U$ supports $C$ in an extreme point $x_\ell$ and all $x_\ell$ lie on the same irreducible component of $\ex_a(C)$ as $x$.
\end{Lem}
By general we mean in this context that the statement is true for all points in a dense (in the Zariski topology) semi-algebraic subset of $\ex_a(C)$.
\begin{proof}
By Straszewicz's Theorem (see Rockafellar \cite[Theorem 18.6]{RockafellarMR0274683}) and the Curve Selection Lemma from semi-algebraic geometry (see Bochnak-Coste-Roy \cite[Theorem 2.5.5]{BochnakMR1659509}), a general extreme point is exposed.
Let $y\in \ex(C)$ be an exposed extreme point contained in a unique irreducible component $Z$ of $\ex_a(C)$ and denote by $\ell_y$ an exposing linear functional. Let $Z_1,\ldots,Z_r$ be the irreducible components of $\ex_a(C)$ labelled such that $Z=Z_1$.
Since the sets $Z_i\cap\partial C\subset C$ are closed, they are compact. Now $\ell_y$ is strictly greater than $-1$ on $Z_i\cap\partial C$ for $i>1$ and therefore, there is a neighbourhood $U$ in $\partial C^o$ of $\ell_y$ such that every $\ell\in U$ is still strictly greater than $-1$ on $Z_i\cap \partial C$. The intersection of this neighbourhood with the semi-algebraic set $S$ of linear functionals exposing extreme points, which is open and dense in the euclidean topology by Proposition \ref{Prop:ExpExtPoi}, is non-empty and open in $\partial C^o$. Pick $\ell_0$ from this open set, then the extreme point $x$ exposed by $\ell_0$ has the claimed properties.
\end{proof}

\begin{Exm}
(a) Again, the above lemma has a simple geometric meaning in the case of polytopes: Every extreme point of the polytope is exposed exactly by the relative interior points of the facet of the dual polytope dual to it, again by Barvinok \cite[Theorem VI.1.3]{BarvinokMR1940576}.\\
(b) In Example \ref{Exm:ExPIncMax}(b), the boundary of the convex set $C$ consists of extreme points and a single $1$-dimensional face. So the only linear functional not exposing an extreme point of $C$ is the dual face to the edge of $C$, which is $(0,-1)\in \ex(C^o)$.
\end{Exm}

By homogenisation, we can prove the analogous version of the above lemma for closed and pointed convex cones.
\begin{Cor}\label{Cor:ExpExtRay}
Let $C\subset \R^{n+1}$ be a pointed closed semi-algebraic convex cone with non-empty interior. Let $F_0\subset C$ be an extreme ray of $C$ such that the line $[F_0]$ is a general point of $\P\exr_a(C)$. 
Let $Z$ be the irreducible component of $\P\exr_a(C)$ with $[F_0]\in Z$. Then there is a supporting hyperplane $\ell_0\in\partial C^\vee$ exposing $F_0$ and a semi-algebraic neighbourhood $U$ of $\ell_0$ in $\partial C^\vee$ such that every $\ell\in U$ supports $C$ in an extreme ray $F_\ell$ of $C$ contained in the regular locus of $Z$, i.e.~$[F_\ell]\in\Xreg[Z]$. \qed
\end{Cor}
The above notion of general now translates into the projective notion, i.e.~the statement is true for points in a dense semi-algebraic subset of the semi-algebraic set of extreme rays as a subset of $\P\exr_a(C)\subset\P^{n}$.

\section{The Algebraic Boundary of Convex Semi-algebraic Sets}\label{sec:AlgBound}
In this section, we consider a full-dimensional closed semi-algebraic convex cone $C\subset \R^{n+1}$ which is pointed, i.e.~it does not contain a line. The algebraic boundary of $C$ is an algebraic cone. In particular, it is the affine cone over its projectivisation, i.e.~$\partial_a C = \wh{\P\partial_a C}$.
The dual convex cone is the set
\[
C^\vee = \{\ell\in\DV[(\R^{n+1})] \colon \forall\;x\in C\; \ell(x)\geq 0\},
\]
i.e.~the set of all half spaces containing $C$. We write $\exr_a(C)$ for the Zariski closure of the union of all extreme rays of $C$ in $\A^{n+1}$. Again, this is an algebraic cone. This is the technically more convenient language for the algebraic duality theory. We will deduce the statements for convex bodies by homogenisation.

We now consider projective dual varieties: Given an algebraic variety $X\subset\P^n$, the dual variety $\DX\subset\DP$ is the Zariski closure of the set of all hyperplanes $[H]\in\DP$ such that $H$ contains the tangent space to $X$ at some regular point $x\in\Xreg$. For computational aspects of projective duality, we refer to Ranestad-Sturmfels \cite{RanSt} and Rostalski-Sturmfels \cite{RostalskiSturm}.
\begin{Prop}\label{Prop:DualAlgBound}
The dual variety to the algebraic boundary of $C$ is contained in the Zariski closure of the extreme rays of the dual convex cone, i.e.
\[
\DX[(\P\partial_a C)]\subset \P\exr_a(C^\vee)
\]
\end{Prop}

\begin{proof}
Let $Y\subset\P\partial_a C$ be an irreducible component of the algebraic boundary of $C$. Let $x\in \wh{Y}\cap\partial C$ be a general point and $H\subset \R^{n+1}$ be a supporting hyperplane to $C$ at $x$. We argue similarly to the proof of Proposition \ref{Prop:ExpExtPoi}: Since $C$ lies in one half-space defined by $H$, so does $\wh{Y}$ locally around $x$. Therefore, $H$ is the tangent hyperplane $\T_x\wh{Y}$. Now the tangent hyperplane to $\wh{Y}$ at $x$ is unique, because $\wh{Y}$ has codimension $1$. So the set of all supporting hyperplanes to $C$ at $x$ is an extreme ray of the dual convex cone. Since $\wh{Y}\cap C$ is Zariski dense in $\wh{Y}$, the hyperplanes tangent to $\wh{Y}$ at points $x\in \wh{Y}\cap C$ are dense in the dual variety to $Y$.
\end{proof}

\begin{Rem}\label{Rem:ExDual}
Let $Z\subset \exr_a(C)$ be an irreducible component. Then the dual variety to $\P Z\subset \P^n$ is a hypersurface in $\DP$, which follows from the biduality theorem in projective algebraic geometry Tevelev \cite[Theorem 1.12]{TevMR2113135}, because $\P Z$ cannot contain a dense subset of projective linear spaces of dimension $\geq 1$. Suppose $\P Z$ contained a dense subset of projective linear spaces of dimension $\geq 1$, then the set $Z\cap \exr(C)$, which is dense in $Z$, would contain a Zariski dense subset of an affine linear space of dimension at least $2$. This contradicts the fact that the set of extreme rays $\exr(C)$ does not contain any line segments other than those lying on the rays themselves.
\end{Rem}

In the language of cones, our first main theorem is the following.
\begin{Thm}\label{Thm:DualIrrComp}
Let $C\subset \R^{n+1}$ be a pointed closed semi-algebraic convex cone with non-empty interior. The dual variety to the locus of extreme rays of $C$ is contained in the algebraic boundary of the dual convex cone $C^\vee$, i.e.
\[
\DX[(\P \exr_a(C))]\subset \P \partial_a C^\vee.
\]
More precisely, the dual variety to every irreducible component of $\P \exr_a(C)$ is an irreducible component of $\P \partial_a C$.
\end{Thm}

\begin{proof}
Let $\P Z\subset \P\exr_a(C)$ be an irreducible component of the locus of extreme rays of $C$. By Corollary \ref{Cor:ExpExtRay}, a general extreme ray $[F_0]\in \P Z\cap (\P \exr(C))$ is exposed by $\ell_0\in\partial C^\vee$ and there is a semi-algebraic neighbourhood $U$ of $\ell_0$ in $\partial C^\vee$ such that every $\ell\in U$ exposes an extreme ray $F_\ell$ of $C$ such that $[F_\ell]\in\Xreg[(\P Z)]$. The hyperplane $\P\ker(\ell)$ is tangent to $\P Z$ at $[F_\ell]$ because $\P Z$ is locally contained in $C$; so $\P U$ is a semi-algebraic subset of $\DX[\P Z]$ of full dimension and the claim follows.
\end{proof}

In the Introduction, we gave an affine version of the preceding theorem that follows from it via homogenisation.
\begin{Cor}\label{Thm:DualIrrCompAffine}
Let $C\subset \R^n$ be a compact convex semi-algebraic set with $0\in\topint(C)$. Let $Z$ be an irreducible component of the Zariski closure of the set of extreme points of its dual convex body. Then the variety dual to $Z$ is an irreducible component of the algebraic boundary of $C$. More precisely, the dual variety to the projective closure $\ol{Z}$ of $Z$ with respect to the embedding $\A^n\to\DP$, $x\mapsto (1:x)$ is an irreducible component of the projective closure of $\partial_a C$ with respect to $\A^n\to\P^n$, $x\mapsto (1:x)$.
\end{Cor}

\begin{proof}
We homogenise the convex body and its dual convex body by embedding both at height $1$ to get convex cones $\co(C)=\{(\lambda,\lambda x)\colon \lambda\geq 0, x\in C\}\subset\R\times\R^n$ and $\co(C^o) = (\co(C))^\vee\subset\R\times\DV[(\R^n)]$. The projective closure $\ol{Z}$ of the irreducible component $Z\subset \ex_a(C^o)$ with respect to the embedding $\A^n\to\DP$, $x\mapsto (1:x)$ is an irreducible component of $\P \exr_a(\co(C)^\vee)$. By the above Theorem \ref{Thm:DualIrrComp}, the dual variety to $\ol{Z}$ is an irreducible component of $\P (\partial_a \co(C))$, which is the projective closure of an irreducible component of the algebraic boundary of $C$ with respect to the embedding $\A^n\to\P^n$, $x\mapsto (1:x)$.
\end{proof}

\begin{Cor}
Let $C\subset \R^{n+1}$ be a pointed closed semi-algebraic convex cone with non-empty interior. We have $\DX[(\P \partial_a C)] =\P \exr_a(C^\vee)$.\qed
\end{Cor}

\begin{Rem}\label{Rem:twoParabolas}
It does not follow from the biduality theorems in both theories that $\DX[(\P \exr_a(C^\vee))]=\P \partial_a C$ simply because the biduality theorem in the algebraic context does not in general apply to this situation, since the varieties in question tend to be reducible. In fact, the mentioned equality does not hold in general, as the following example shows: 
Let $C\subset\R^2$ be the convex set defined by the inequalities $x^2+y^2-1\geq 0$ and $x\leq 3/5$, see Figure \ref{fig:CirclePlaneDuality}. The dual convex body is the convex hull of the set $\{(x,y)\in\R^2\colon x^2+y^2-1\geq 0, x\geq -3/5\}$ and the point $(-5/3,0)$ (it cannot be defined by simultaneous polynomial inequalities, i.e.~it is not a basic closed semi-algebraic set). Its algebraic boundary has three components, namely the circle and the two lines $y=3/4x+5/4$ and $y = -3/4 x-5/4$. The set of extreme points of $C$ is $\{(x,y)\colon x^2+y^2-1=0,x\leq 3/5\}$. So $\ex_a(C) = \V(x^2+y^2-1)$ and $\V(x^2+y^2-1)^\ast = \V(x^2+y^2-1)\subsetneq \partial_a C^o$.
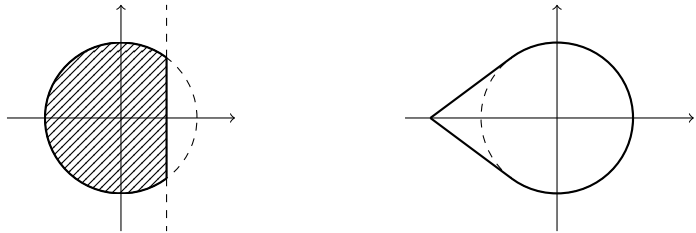
\begin{figure}[h]
\begin{center}
\begin{minipage}[]{0.4\textwidth}
\begin{center}
\begin{tikzpicture}
\draw[->] (-5.5,0) -- (-2.5,0);
\draw[->] (-4,-1.5) -- (-4,1.5);
\draw[thick] (-3.4,-0.8) -- (-3.4,0.8);
\draw[dashed] (-3.4,0.8) -- (-3.4,1.5);
\draw[dashed] (-3.4,-0.8) -- (-3.4,-1.5);
\draw[dashed] (-4,0) circle(1);
\clip (-3.4,-1) rectangle (-5,1);
\filldraw[pattern=north east lines] (-4,0) circle(1);
\draw[thick] (-4,0) circle(1);
\end{tikzpicture}
\end{center}
\end{minipage}
\begin{minipage}[]{0.4\textwidth}
\begin{center}
\begin{tikzpicture}
\clip (-2,2) rectangle (2,-2);
\draw[->] (-2,0) -- (1.8,0);
\draw[->] (0,-1.5) -- (0,1.5);
\draw[thick] (-1.666666,0) -- (-0.6,0.8);
\draw[thick] (-1.666666,0) -- (-0.6,-0.8);
\draw[dashed] (0,0) circle(1);
\clip (-0.6,-2) rectangle (2,2);
\draw[thick] (0,0) circle(1);
\end{tikzpicture}
\end{center}
\end{minipage}
\end{center}
\caption{A circle cut by a halfspace and its dual convex body.}
\label{fig:CirclePlaneDuality}
\end{figure}
\end{Rem}

The following statement gives a complete semi-algebraic characterisation of the irreducible subvarieties $Y\subset\exr_a(C)$ with the property that $\DX[Y]$ is an irreducible component of the algebraic boundary of $C^\vee$.
\begin{Thm}\label{Thm:DualCond}
Let $C\subset \R^{n+1}$ be a pointed closed semi-algebraic convex cone. Let $Z$ be an irreducible algebraic cone contained in $\exr_a(C)$ and suppose $Z\cap \exr(C)$ is Zariski dense in $Z$. Then the dual variety to $\P Z$ is an irreducible component of $\P \partial_a C^\vee$ if and only if the dimension of the normal cone to a general point $x\in Z\cap \exr(C)$ is equal to the codimension of $Z$, i.e.
\[
\dim(Z)+\dim(N_C(\R_+x)) = n+1.
\]
Conversely, if $Y$ is an irreducible component of the algebraic boundary of $C^\vee$, then the dual variety to $\P Y$ is an irreducible subvariety of $\P \exr_a(C)$, the set $\DX[(\P Y)]\cap \exr(C)$ is Zariski dense in $\DX[(\P Y)]$ and the above condition on the normal cone is satisfied at a general extreme ray for the affine cone over $\DX[(\P Y)]$.
\end{Thm}

To be clear, the normal cone is $N_C(\R_+x) = \{\ell\in\DV[(\R^{n+1})] \colon \forall\; y\in C\setminus \R_+x\; \ell(y) \geq \ell(x)= 0\}$.
\begin{proof}
Consider the semi-algebraic set $\Sigma\subset\partial C\times\partial C^\vee\subset \R^{n+1}\times \DV[(\R^{n+1})]$ defined as
\[
\Sigma=\{(x,\ell)\in \R^{n+1}\times \DV[(\R^{n+1})]\colon x\in \Xreg[Z]\cap \exr(C), \ell\in C^\vee, \ell(x)=0\}
\]
This is the set of all tuples $(x,\ell)$, where $x$ spans an extreme ray of $C$ and is a regular point of $Z$ and $\ell$ is a supporting hyperplane to $C$ at $x$, i.e.~the fibre of the projection $\pi_1$ onto the first factor over a point $x$ is the normal cone $N_C(\R_+ x)$. Since a supporting hyperplane to $C$ at $x$ is tangent to $Z$ at $x$, this bihomogeneous semi-algebraic incidence correspondence is naturally contained in the conormal variety $\CN[\P Z]\subset \P^n\times\DP$ of the projectivisation of $Z$.
Now the image $\pi_2(\Sigma)$ is Zariski dense in $\DX[\P Z]$ if and only if $\DX[\P Z]$ is an irreducible component of the projectivisation of the algebraic boundary of $C^\vee$. Indeed, $\pi_2(\Sigma)\subset \DX[\P Z]\cap \P \partial C^\vee$ and so if it is dense in $\DX[\P Z]$, we immediately get that $\DX[\P Z]\subset\P \partial_a C^\vee$ is an irreducible component, because $\DX[\P Z]$ is a hypersurface (cf.~Remark \ref{Rem:ExDual}(b)). Conversely, we have seen in the proof of the above proposition that if $\DX[\P Z]\subset\P \partial_a C^\vee$ is an irreducible component, the unique tangent hyperplane to a general point of $\DX[\P Z]\cap \P \partial C^\vee$ spans an extreme ray of $C$, i.e.~a general point of $\DX[\P Z]\cap \P\partial C^\vee$ is contained in $\pi_2(\Sigma)$.

This says, that $\Sigma$ is dense in $\CN[\P Z]$, i.e.~$\dim(\Sigma)=\dim(\CN[\P Z])+2=n+1$ if and only if $\DX[\P Z]$ is an irreducible component of $\P \partial_a C^\vee$.

On the other hand, counting dimensions of $\Sigma$ as the sum of the dimensions of $Z$ and the dimension of the fibre over a general point in $\Xreg[Z]\cap\exr(C)$, we see that $\dim(\Sigma)=n+1$ if and only if the claimed equality of dimensions
\[
\dim(Z)+\dim(N_C(\R_+x)) = n+1
\]
holds.
The second part of the statement follows from the first by Proposition \ref{Prop:DualAlgBound}.
\end{proof}

\begin{Rem}
We want to compare this theorem to the result of Ranestad and Sturmfels in \cite{RanSt}: They consider the convex hull of a smooth algebraic variety $X\subset\P^n$ and make the technical assumption that only finitely many hyperplanes are tangent to the variety $X$ in infinitely many points, which is needed for a dimension count in the proof. We get rid of this technical assumption in the above theorem. The assumption that the extreme rays are Zariski dense in the variety $Z$ in question, compares best to the Ranestad-Sturmfels assumption. It is semi-algebraic in nature.
\end{Rem}

The corresponding affine statement to Theorem \ref{Thm:DualCond} is the following. We take projective closures with respect to the same embeddings as in the affine version Corollary \ref{Thm:DualIrrCompAffine} of Theorem \ref{Thm:DualIrrComp} above.
\begin{Cor}\label{Thm:DualCondAffine}
Let $C\subset \R^n$ be a compact convex semi-algebraic set with $0\in\topint(C)$. Let $Z$ be an irreducible subvariety of $\ex_a(C)$ and suppose $Z\cap \ex(C)$ is dense in $Z$. Then the dual variety to $\ol{Z}$ is an irreducible component of $\ol{\partial_a C^o}$ if and only if
\[
\dim(Z) + \dim(N_C(\{x\})) = n
\]
for a general extreme point $x\in Z\cap \ex(C)$.
Conversely, if $Y$ is an irreducible component of $\partial_a C^o$, then the dual variety to $\ol{Y}$ is an irreducible subvariety of $\ol{\ex_a(C)}$, the set $\DX[\ol{Y}]\cap \ex(C)$ is dense in $\DX[\ol{Y}]$ and the condition on the normal cone is satisfied at a general extreme point.
\end{Cor}

\begin{proof}
Again, the proof is simply by homogenising as above. Note that the dimension of the normal cone does not change when homogenising.
\end{proof}

In the following affine examples we will drop the technical precision of taking projective closures and talk about the dual variety to an affine variety to make them more readable.
\begin{Exm}
Let $C=\{x\in\R^n\colon g_1(x)\geq 0,\ldots,g_r(x)\geq 0\}\subset\R^n$ be a basic closed semi-algebraic convex set with non-empty interior defined by $g_1,\ldots,g_r\in\R[x_1,\ldots,x_n]$. Then the algebraic boundary $\partial_a C$ is contained in the variety $\V(g_1)\cup\ldots\cup\V(g_r) = \V(p_1)\cup\ldots\V(p_s)$, where $p_1,\ldots,p_s$ are the irreducible factors of the polynomials $g_1,\ldots,g_r$. The irreducible hypersurface $\V(p_i)$ is an irreducible component of $\partial_a C$ if and only if $\V(p_i)\cap \partial C$ is a semi-algebraic set of codimension $1$. By the above Corollary \ref{Thm:DualCondAffine}, we can equivalently check the following conditions on the dual varieties $X_i$ to the projective closure $\ol{\V(p_i)}$:
\begin{compactitem}
\item The extreme points of the dual convex set are dense in $X_i$ via $\R^n\to\DP$, $x\mapsto (1:x)$.
\item A general extreme point of the dual convex set in $X_i$ exposes a face of $C$ of dimension $\codim(X_i)-1$.
\end{compactitem}
We consider the convex set shown in Figure \ref{fig:Schleife}, whose algebraic boundary is the cubic curve $X = \V(y^2-(x+1)(x-1)^2)$, with different descriptions as a basic closed semi-algebraic set.
\begin{figure}[h]
\begin{center}
\begin{minipage}[]{0.4\textwidth}
\begin{center}
\begin{tikzpicture}
\filldraw[black!30!white,domain=-1.4142:1.4142,variable=\t,smooth] plot (\t*\t-1,2*\t-\t*\t*\t);
\draw[black,variable=\t,domain=-1.65:1.65] plot (\t*\t-1,2*\t-\t*\t*\t);
\end{tikzpicture}
\end{center}
\end{minipage}
\begin{minipage}[]{0.4\textwidth}
\begin{center}
\includegraphics[scale =0.6]{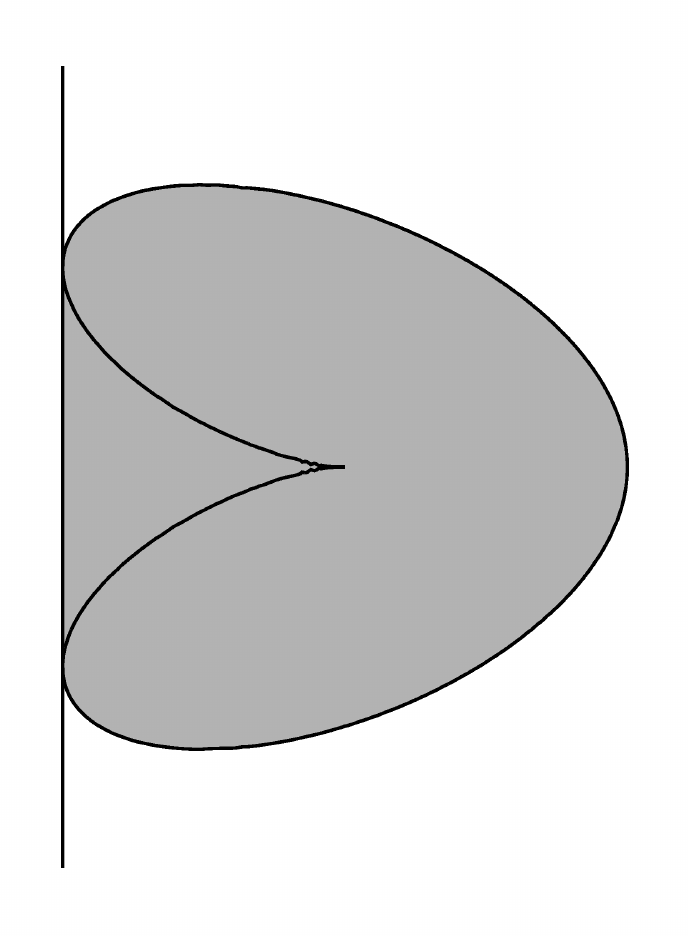}
\end{center}
\end{minipage}
\end{center}
\caption{A basic closed semi-algebraic set in the plane on the left and its dual convex set on the right.}
\label{fig:Schleife}
\end{figure}
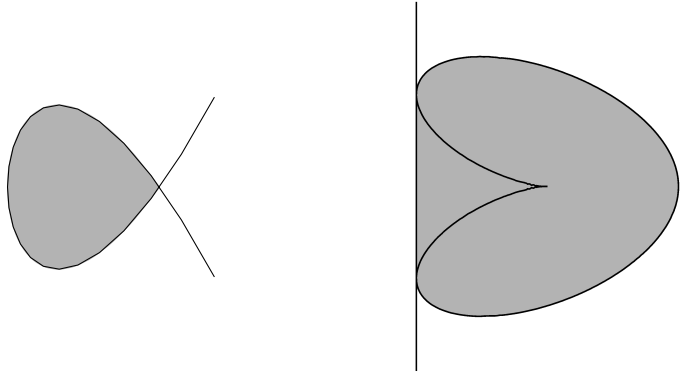
The dual convex body is the convex hull of a quartic curve. Its algebraic boundary is
\[
\V(4x^4+32y^4+13x^2y^2-4x^3+18xy^2-27y^2 )\cup\V(x+1).
\]
Here, the line $\V(x+1)$ is a bitangent to the quartic and the dual variety of the node $(1,0)$ of the cubic and the quartic is the dual curve to the cubic.
We define $C$ using the cubic inequality and additionally either one linear inequality or the two tangents to the branches of $X$ in $(1,0)$
\begin{eqnarray*}
C & = & \{(x,y)\in\R^2\colon y^2-(x+1)(x-1)^2\leq 0, x\leq 1\} \\
  & = & \{(x,y)\in\R^2\colon y^2-(x+1)(x-1)^2\leq 0, y\geq \sqrt{2}(x-1),y\leq -\sqrt{2}(x-1)\},
\end{eqnarray*}
and we see both conditions in action. First, the dual variety to the affine line $x=1$ is $(-1,0)$, which is not an extreme point of $C^o$. The first condition mentioned above shows, that the line $\V(x-1)$ corresponding to the second inequality in the first description is not an irreducible component of $\partial_a C$.
 In the second description, the dual variety to the affine line $y = \sqrt{2}(x-1)$ is the point $P = (-1,\frac{1}{\sqrt{2}})$, which is an extreme point of $C^o$. The normal cone $N_{C^o}(\{P\})$ is $1$-dimensional, because the supporting hyperplane is uniquely determined - it is the bitangent $\V(x+1)$ to the quartic. So by the second condition above, the line $\V(y-\sqrt{2}(x-1))$ is not an irreducible component of $\partial_a C$.
\end{Exm}

\begin{Cor}\label{Cor:AlgBoundLocExtPoi}[to Corollary \ref{Thm:DualCondAffine}]
Let $C\subset \R^{n}$ be a compact semi-algebraic set with $0\in\topint(C)$.
Let $Y\subset \partial_a C^o$ be an irreducible component such that $\DX[\ol{Y}]\subset \ol{\ex_a(C)}$ is not an irreducible component. If $\DX[\ol{Y}]$ is contained in a bigger irreducible subvariety $Z\subset \ol{\ex_a(C)}$ such that $Z\cap\ex(C)$ is dense in $Z$, then
\begin{compactitem}
\item $\DX[\ol{Y}]\subset \Xsing[Z]$ or 
\item $\DX[\ol{Y}]$ is contained in the algebraic boundary of the semi-algebraic subset $\ex(C) \cap Z$ of $Z$.
\end{compactitem}
\end{Cor}

\begin{proof}
Let $Z\subset \ol{\ex_a(C)}$ be an irreducible subvariety.
If $\ell\in (\R^{n})^\ast$ defines a supporting hyperplane to an extreme point $x\in \ex(C)$ that is an interior point of the semi-algebraic set $\ex(C)\cap Z$ as a subset of $Z$ and $(1:x)\in\Xreg[Z]$, then the variety $Z$ lies locally in one of the half spaces defined by $(1:\ell)$ and therefore $(1:\ell)$ is tangent to $Z$ at $(1:x)$. In particular, the dimension of the normal cone $N_C(\{x\})$ is bounded by the local codimension of $Z$ at $(1:x)$. Now if $\DX[\ol{Y}]$ is strictly contained in $Z$, it cannot contain $(1:x)$ by Corollary \ref{Thm:DualCondAffine} because $\dim(\DX[\ol{Y}])<\dim(Z)$.
\end{proof}

The set $Z\cap \ex(C)$ in the above corollary does not need to be a regular semi-algebraic set. So the second condition can also occur in the following way. 
\begin{Exm}
Consider the convex hull $C$ of the half ball $\{(x,y,z)\in\R^3\colon x^2+y^2+z^2\leq 1, x\geq 0\}$ and the circle $X = \{(x,y,z)\in\R^3\colon x^2+y^2\leq 1, z=0\}$. The Zariski closure of the extreme points of $C$ is the sphere $S^2$. Every point of the circle $X$ is a regular point of $S^2$ and $X$ is contained in the algebraic boundary of $\ex(C)\cap S^2\subset S^2$, because the semi-algebraic set $\ex(C) \cap S^2$ does not have local dimension $2$ at the extreme points $(x,y,0)\in X\cap \ex(C)$ where $x<0$.
The algebraic boundary of the dual convex set has three irreducible components, namely the sphere $S^2$ and the dual varieties to the two irreducible components $X$ and $\V(y^2+z^2-1,x)$ of $\partial_a(\ex(C)\cap S^2)\subset S^2$.
\end{Exm}

The following examples show how the statement of the corollary can be used to determine the algebraic boundary in concrete cases. 
\begin{Exm}
Consider the spectrahedron $P = \{(x,y,z)\in\R^3\colon Q(x,y,z) \geq 0\}$ where $Q$ is the symmetric matrix
\[
Q = \left(
\begin{array}[h]{cccc}
1 & x & 0 & x \\
x & 1 & y & 0 \\
0 & y & 1 & z \\
x & 0 & z & 1 
\end{array}\right),
\]
studied by Rostalski and Sturmfels in \cite[Section 1.1]{RostalskiSturm} and called pillow. The Zariski closure of the set of extreme points of $P$ is defined by the equation $\det(Q)=0$, where
\[
\det(Q) = x^2(y-z)^2-2x^2-y^2-z^2+1.
\]
The algebraic boundary of the dual convex body $P^o$ is the hypersurface
\begin{eqnarray*}
\partial_a P^o = & \V(b^2+2bc+c^2-a^2b^2-a^2c^2-b^4-2b^2c^2-2bc^3-c^4-2b^3c)\cup \\
 & \V(2-a^2+2ab-b^2+2bc-c^2-2ac)\cup\\
& \V(2-a^2-2ab-b^2+2bc-c^2+2ac),
\end{eqnarray*}
computed in Rostalski-Sturmfels \cite[Section 1.1, Equations 1.7 and 1.8]{RostalskiSturm}. The first quartic is the dual variety to the quartic $\V(\det(Q))$. The two quadric hypersurfaces are products of linear forms over $\R$ and they are the dual varieties to the four corners of the pillow, namely $\frac{1}{\sqrt{2}}(1,1,-1)$, $\frac{1}{\sqrt{2}}(-1,-1,1)$, $\frac{1}{\sqrt{2}}(1,-1,1)$ and $\frac{1}{\sqrt{2}}(-1,1,-1)$. These four points are extreme points of $P$ and singular points of $\V(\det(Q))$.\\
\end{Exm}

Another interesting consequence of Corollary \ref{Thm:DualCondAffine} concerns the semi-algebraic set $\ex(C)$.
\begin{Cor}\label{Cor:ExtPointsCentral}
Let $C\subset\R^n$ be a compact semi-algebraic convex set with $0\in\topint(C)$.
Every extreme point $x$ of $C$ is a central point of the dual variety of at least one irreducible component of $\ol{\partial_a C^o}$ via $\A^n\to\P^n$, $x\mapsto (1:x)$.
\end{Cor}

A point $x$ on a real algebraic variety $X\subset\P^n$ is called central if $X(\R)$ has full local dimension around $x$. Equivalently, $x\in X$ is central if it is the limit of a sequence regular real points of $X$, cf.~Bochnak-Coste-Roy \cite[Section 7.6 and Proposition 10.2.4]{BochnakMR1659509}.
\begin{proof}
By Straszewicz's Theorem \cite[Theorem 18.6]{RockafellarMR0274683}, it suffices to prove, that the statement holds for exposed extreme points because every extreme point is the limit of an exposed one. So let $x$ be an exposed extreme point of $C$ and let $F_x =\{\ell\in C^o \colon \ell(x)=-1\}$ be the dual face. Because $x$ is exposed, the normal cone $N_{C^o}(F_x)=\R_+ x$ is $1$-dimensional. Fix a relative interior point $\ell\in F_x$. Let $Y$ be an irreducible component of $\partial_a C^o$ on which $\ell$ is a central point and let $(\ell_j)_{j\in\N}\subset \Xreg[Y](\R)$ be a sequence of regular real points converging to $\ell$ in the euclidean topology. There is a unique (up to scaling) linear functional minimising in $\ell_j$ over $C^o$, namely $y_j\in\partial C$ with $\ell_j(y_j)=-1$ and $\alpha_j(y_j)=-1$ for all $\alpha\in T_{\ell_j}Y$. Since $(y_j)$ is a sequence in a compact set, there exists a converging subsequence; without loss of generality, we assume that $(y_j)_{j\in\N}$ converges and we call the limit $y$. Note that $y$ represents a central point of $\DX[\ol{Y}]$. We know $y\in\partial C$ and
\[
\ell(y) = \lim_{j\to\infty} \ell_j(y) = \lim_{j\to\infty} \ell_j(\lim_{k\to\infty} y_k) =-1,
\]
so $y$ exposes the face $F_x$ of $C^o$ and is therefore equal to $x$ by $N_{C^o}(F_x) = \R_+ x$.
\end{proof}

We take a short look at implications of this corollary to hyperbolicity cones.
\begin{Exm}
A homogeneous polynomial $p\in\R[x_0,\ldots,x_n]$ of degree $d$ is called hyperbolic with respect to $e\in\R^{n+1}$ if $p(e)\neq 0$ and the univariate polynomial $p(te - x)\in\R[t]$ has only real roots for every $x\in\R^{n+1}$. We consider the set
\[
C_p(e) = \{x\in\R^{n+1}\colon \textrm{all roots of } p(te-x) \textrm{ are non-negative}\},
\]
which is called the hyperbolicity cone of $p$ (with respect to $e$). It turns out to be a convex cone, cf.~\cite{RenegarMR2198215}. Assume that all non-zero points in the boundary of $C_p(e)$ are regular points of $\V(p)$. Then by Corollary \ref{Cor:AlgBoundLocExtPoi} the algebraic boundary of the dual convex cone is the dual variety to $\V(q)$ where $q$ is the unique irreducible factor of $p$ which vanishes on $\partial C_p$.

The assumption on the hyperbolicity cone being smooth is essential: Consider the hyperbolicity cone of $p = y^2z-(x+z)(x-z)^2\in\R[x,y,z]$ with respect to $(0,0,1)$. The cubic $\V(p)\subset\R^3$ is singular along the line $\R(1,0,1)$ and the algebraic boundary of the dual convex cone has an additional irreducible component, namely the hyperplane dual to this line because the normal cone has dimension $2$ at this extreme ray, see Figure \ref{fig:Schleife}.

Let now $C_p(e)$ be any hyperbolicity cone and decompose $\partial_a C_p(e) = X_1\cup \ldots\cup X_r$ into its irreducible components $X_1,\ldots,X_r$. The dual convex cone $C_p(e)^\vee$ is the conic hull of the regular real points of the dual varieties of the irreducible components $X_i$ up to closure, i.e.
\[
C_p(e)^\vee = {\rm cl}(\co( (\DX[X_1])_{\rm reg}(\R) \cup \ldots \cup (\DX[X_r])_{\rm reg}(\R) )).
\]
Indeed, the right hand side contains every central point of every variety $\DX[X_i]$ and by Corollary \ref{Cor:ExtPointsCentral}, this gives one inclusion. Conversely, let $\ell$ be a general real point of $\DX[X_i]$ for any $i$. Then $\ell$ is tangent to $X_i$ in a regular real point of $\partial_a C_p(e)$ and by hyperbolicity of $p$, the linear functional has constant sign on the hyperbolicity cone $C_p(e)$ because every line through the hyperbolicity cone intersects every regular real point of $\partial_a C_p(e)$ with multiplicity $1$, cf.~Plaumann-Vinzant \cite[Lemma 2.4]{PlauVinMR3066450}.
\end{Exm}

How can we compute these exceptional varieties of extreme points? Given the algebraic boundary of the dual convex set, the following theorem gives an answer.
In its statement, we use an iterated singular locus: The $k$-th iterated singular locus of a variety $X$, denoted by $X_{k,{\rm sing}}$, is the singular locus of the $(k-1)$ iterated singular locus. The $1$-st iterated singular locus is the usual singular locus of $X$.
\begin{Thm}\label{Thm:comp}
Let $C\subset \R^n$ be a compact semi-algebraic convex set with $0\in \topint(C)$ and suppose that every point $\ell\in\partial C^o$ is a regular point on every irreducible component of $\partial_a C^0$ containing it. Let $Z\subset \ex_a(C^o)$ be an irreducible subvariety such that $\DX[\ol{Z}]$ is an irreducible component of $\ol{\partial_a C}$. If $\codim(Z) = 1$, then $Z$ is an irreducible component of $\partial_a C^o$.
If $\codim(Z) = c > 1$, then $Z$ is an irreducible component of an iterated singular locus, namely it is an irreducible component of one of the varieties $\Xsing[(\partial_a C^o)], (\partial_a C^o)_{2,{\rm sing}},\ldots, (\partial_a C^o)_{c-1, {\rm sing}}$.
\end{Thm}

\begin{proof}
Assume $\codim(Z)=c>1$ and let $\ell\in Z\cap \ex(C^o)$ be a general point. Since Whitney's condition a is satisfied for $(\Xreg,Z)$ at $\ell$ for every irreducible component $X\subset \partial_a C^o$ with $Z\subset X$ by Bochnak-Coste-Roy \cite[Theorem 9.7.5]{BochnakMR1659509}, every extreme ray $\R_+ x$ of $N_{C^o}(\{\ell\})$ is tangent to $Z$ at $\ell$ by Corollary \ref{Cor:ExtPointsCentral}. Since the extreme rays of the normal cone $N_{C^o}(\{\ell\})$ span the smallest linear space containing it, the dimension of $Z$ is bounded from above by $\codim(N_{C^o}(\{\ell\}))$. The assumption that $\DX[\ol{Z}]$ is an irreducible component of $\ol{\partial_a C}$ implies $\dim(Z) = \codim(N_{C^o}(\{\ell\}))$ by Corollary \ref{Thm:DualCondAffine}.
It follows that the tangent space $\T_\ell Z$ is the lineality space of the convex cone $N_{C^o}(\{\ell\})^\vee$.
To show that $Z$ is an irreducible component of $(\partial_a C^o)_{j,{\rm sing}}$, suppose $Y\subset (\partial_a C^o)_{k,{\rm sing}}$ is an irreducible component with $Z\subsetneq Y$ and $\Xreg[Y]\cap Z\neq \emptyset$ and let $\ell\in Z\cap \ex(C^o)$  be a general point with $\ell \in \Xreg[Y]$. Then $\T_\ell Z\subsetneq \T_\ell Y$ and there is an extreme ray $\R_+ x$ of $N_{C^o}(\ell)$ with $x\in\ex(C)$ and $x\vert_{\T_\ell Y} \neq 0$. By Corollary \ref{Cor:ExtPointsCentral}, there is an irreducible component $X\subset \partial_a C^o$ such that $x$ is a central point of $\DX$. So by assumption, $\ell\in\Xreg$ and $x\in (\T_\ell X)^\perp$. Since $x\vert_{\T_\ell Y}\neq 0$, the varieties $Y$ and $X$ intersect transverally at $\ell$. So $Z\subset Y\cap X \subsetneq Y$ and $Y\cap X\subset (\partial_a C^o)_{j,{\rm sing}}$ is an irreducible component for some $j>k$ because the multiplicity of a point in $Y\cap X$ in $\partial_a C^o$ is higher than the multiplicity of a general point on $Y$. Induction on the codimension of $Z$ proofs the theorem.
\end{proof}

\begin{Rem}
(a) This theorem gives a computational way to get a list of candidates for the dual varieties to irreducible components of the algebraic boundary of $C$, given the algebraic boundary of $C^o$. Certain of these candidates may fail to contribute an irreducible component due to semi-algebraic constraints. For illustration, we will apply it to two examples.\\
(b) The assumption that all irreducible components of $\partial_a C^o$ are smooth along the boundary of $C^o$ is used to show that the stratification into iterated singular loci is sufficient in this case. In general, it may be necessary to refine this stratification such that Whitney's condition a is satisfied for all adjacent strata, see Example \ref{Exm:whitney}.
\end{Rem}

\begin{Exm}[cf.~Remark \ref{Rem:twoParabolas}]
We consider the convex set $C\subset\R^2$ in the plane defined by the two inequalities $x^2+y^2\leq 1$ and $x\leq 3/5$, see Figure \ref{fig:CirclePlaneDuality}. Its algebraic boundary is the plane curve $\V( (x^2+y^2-1)(x-3/5))$.
The dual convex body is the convex hull of the set $\{(X,Y)\in\R^2 \colon X^2+Y^2\leq 1, X\geq -3/5\}$ and the point $(-5/3,0)$. Its algebraic boundary is the curve $\partial_a C^o = \V( (X^2+Y^2-1)(4Y-3X-5)(4Y+3X+5))$. Its three irreducible components are smooth and its singular locus consists of three points, namely $(-5/3,0)$ and $(-3/5,\pm 4/5)$.
By the above theorem, a complete list of candidates for the algebraic boundary of $C$ are the dual varieties to the circle $\V(X^2+Y^2-1)$ and the irreducible components of the first iterated singular locus, i.e.~the lines dual to the points $(-5/3,0)$ and $(-3/5,\pm 4/5)$. In fact, the last two points do not contribute an irreducible component to $\partial_a C$, because the normal cone to $C^o$ at these points is $1$-dimensional, cf.~Corollary \ref{Thm:DualCondAffine}.\\
We can also look at it dually and compute the algebraic boundary $\partial_a C^o$ from the singularities of the algebraic boundary of $C$: The curve $\partial_a C$ is reducible, all components are smooth, and its singular locus consists of two points, namely $(3/5,\pm 4/5)$. Both of these points dualize to irreducible components of $\partial_a C^o$.
\end{Exm}

\begin{Exm}
As an example in $3$-space, consider the convex set $C$ defined as the intersection of two affinely smooth cylinders given by the inequalities $x^2+y^2\leq 1$ and $3y^2 + 4z^2 - 4y \leq 4$. The algebraic boundary of $C$ is the (reducible) surface $\V( (x^2+y^2-1)(3y^2+4z^2-4y-4))$, whose singular locus is a smooth curve of degree $4$, namely the intersection of the two cylinders. Since the dual varieties to the cylinders are curves and the iterated singular loci of $\partial_a C$ are this smooth curve of degree $4$ or empty, the algebraic boundary of the dual convex body is, by Theorem \ref{Thm:comp}, the dual variety of this curve, which is a surface of degree $8$ defined by the polynomial
{\small
\begin{eqnarray*}
& & -240 X^{8}-608 X^{6} Y^{2}-240 X^{4} Y^{4}+384 X^{2} Y^{6}+256 Y^{8}+840 X^{6} Z^{2} \\
& & +696 X^{4} Y^{2} Z^{2}-192 X^{2} Y^{4} Z^{2}+384 Y^{6} Z^{2}-1215 X^{4} Z^{4}+696 X^{2} Y^{2} Z^{4} \\
& & -240 Y^{4} Z^{4} +840 X^{2} Z^{6}-608 Y^{2} Z^{6}-240 Z^{8}-896 X^{6} Y-2304 X^{4} Y^{3} \\
& & -1920 X^{2} Y^{5}-512 Y^{7}+1152 X^{4} Y Z^{2}+192 X^{2} Y^{3} Z^{2} +768 Y^{5} Z^{2}-1848 X^{2} Y Z^{4} \\
& & +2784 Y^{3} Z^{4}+1504 Y Z^{6}+832 X^{6}+1312 X^{4} Y^{2}-160 X^{2} Y^{4} -640 Y^{6}-984 X^{4} Z^{2} \\
& & -4144 X^{2} Y^{2} Z^{2}-3520 Y^{4} Z^{2}-234 X^{2} Z^{4}-2504 Y^{2} Z^{4}+232 Z^{6} \\
& & +2176 X^{4} Y+3584 X^{2} Y^{3}+1408 Y^{5}+2048 X^{2} Y Z^{2}+576 Y^{3} Z^{2}-1640 Y Z^{4} \\
& & -800 X^{4}-288 X^{2} Y^{2}+656 Y^{4}-424 X^{2} Z^{2}+2808 Y^{2} Z^{2}+313 Z^{4}-1664 X^{2} Y \\
& & -1280 Y^{3}-128 Y Z^{2}+64 X^{2}-416 Y^{2}-456 Z^{2}+384 Y+144.
\end{eqnarray*}
}
Viewed dually, this example is more complicated. The algebraic boundary of $C^o$ is the surface of degree $8$ defined by the above polynomial, which has singularities along the boundary of $C^o$. So the above theorem is not applicable in this case but the conclusion is still true and we compute the iterated singular loci for demonstration.
The singular locus of the surface has $4$ irreducible components: the dual varieties to the cylinders, which are circles, namely $\V(Z,X^2+Y^2-1)$ and $\V(X,4Y^2+4Z^2-4Y-3)$, a complex conjugate pair of quadrics $\V(2 Y^{2}-Y+2, 4 X^{2}-3 Z^{2}-2 Y Z^{2}+8 Y-4)$, and a curve of degree $12$, which we denote by $X_{12}$.
 The second iterated singular locus, which is the singular locus of the union of these $4$ irreducible curves, consists of $24$ points. $16$ of them are the singular points of $X_{12}$ and the other $8$ points are intersection points of $X_{12}$ with the complex conjugate pair of quadrics $\V(2 Y^{2}-Y+2, 4 X^{2}-3 Z^{2}-2 Y Z^{2}+8 Y-4)$. The two circles dual to the cylinders intersect the curve $X_{12}$ only in singular points of the latter. There are no other intersection points of the irreducible components of $\Xsing[(\partial_a C^o)]$. Of these $24$ points in $(\partial_a C^o)_{2,{\rm sing}}$ only $4$ are real. They are $(\pm \sqrt{5/9},2/3,0)$ and $(0,-1/6,\pm \sqrt{5/9})$.
Now the difficult job is to exclude those varieties that do not contribute irreducible components to the algebraic boundary of $C$. The dual variety to $\partial_a C^o$ is only a curve, so it cannot be an irreducible component of $\partial_a C$. Next, we discuss the irreducible components of $\Xsing[(\partial_a C^o)]$: The dual varieties to the complex conjugate pair of quadrics cannot be an irreducible component of $\partial_a C$ either, because the real points will not be dense in this hypersurface. Why the dual variety to the curve $X_{12}$ is not an irreducible component of $\partial_a C$ is not obvious. Of the irreducible components of $(\partial_a C^o)_{2,{\rm sing}}$, the $4$ real points must be considered as potential candidates for dual varieties to irreducible components of $\partial_a C$.
\end{Exm}

To close, we want to consider an example of a convex set whose algebraic boundary is not smooth along its euclidean boundary and for which the conclusion of the Theorem \ref{Thm:comp} is false. As remarked above, the stratification into iterated singular loci must be refined to a stratification that is Whitney a-regular.
\begin{Exm}\label{Exm:whitney}
Consider the surface in $\A^3$ defined by
\[
f = (z^2+y^2 - (x+1)(x-1)^2)(y-5(x-1))(y+5(x-1)),
\]
which is the union of an irreducible cubic and two hyperplanes meeting along the line $\V(x-1,y)$. The cubic surface is a rotation of the nodal curve shown in Figure \ref{fig:Schleife} on the left along the $x$-axis, so the convex set $C$ bounded by the cubic looks like a teardrop. We consider the extreme point $p=(1,0,0)$ of $C$: The normal cone is two-dimensional and so the dual hyperplane $p^\perp$ is an irreducible component of the algebraic boundary of $C^o$. Indeed, the point $p$ is a singular point of the cubic that lies on the line $\V(x-1,y)$, which is an irreducible component of the singular locus of the reducible surface $\V(f)$, so $p$ cannot be found by computing the iterated singular loci of $\V(f)$.
We make this discussion relevant by perturbing the above polynomial $f$ in such a way that it becomes irreducible and shows the same behaviour: Consider the polynomial
\[
g = f + \frac{1}{10} (x-1)yz^2,
\]
which is irreducible over $\Q$. The surface $\V(g)\subset \A^3$ is the algebraic boundary of a convex set $C'$, a perturbation of the teardrop $C$. Convexity of $C'$ can be checked by writing $z$ as a function of $x$ and $y$ and checking its convexity resp.~concavity using its Hessian matrix (note that $z$ only occurs to the power of $2$ in $g$).
The point $p$ is also an extreme point of $C'$ and the normal cone at $p$ relative to $C'$ is still $2$-dimensional. Yet the algebraic boundary of $C'$ is only singular along the line $\V(x-1,y)$, which is a smooth curve. So we don't find $\{p\}$ as an irreducible component of an iterated singular locus of $\partial_a C'=\V(g)$.

Note that Whitney's condition a for $(\V(g),\V(x-1,y))$ is not satisfied at $p$ because a hyperplane that is in limiting position for supporting hyperplanes to the teardrop $C'$ do not contain the line $\V(x-1,y)$. Refining the stratification of iterated singular loci into a Whitney a-regular stratification would detect this special extreme point.
\end{Exm}

\textbf{Acknowledgements.} This work is part of my PhD thesis. I would like to thank my advisor Claus Scheiderer for his encouragement and support, the Studienstiftung des deutschen Volkes for their financial and ideal support of my PhD project, and the National Institute of Mathematical Sciences in Korea, which hosted me when I finished this paper.

\end{document}